\documentclass[draft,onecolumn]{IEEEtran}
\IEEEoverridecommandlockouts
\usepackage{cite}
\usepackage{amsmath,amssymb,amsfonts,amsthm}
\usepackage{algorithmic}
\usepackage{graphicx}
\usepackage{textcomp}
\usepackage{xcolor}
\def\BibTeX{{\rm B\kern-.05em{\sc i\kern-.025em b}\kern-.08em
    T\kern-.1667em\lower.7ex\hbox{E}\kern-.125emX}}
\newtheorem{thm}{Theorem}
\newtheorem{defn}{Definition}

\newtheorem{lem}{Lemma}

\usepackage{geometry}
\usepackage{setspace}
\begin{document}

\title{ Quaternion Windowed Linear Canonical Transform of Two-Dimensional Signals}
\author{\IEEEauthorblockN{Wen-Biao, Gao$^{1,2}$,
		Bing-Zhao, Li$^{1,2,\star}$,}\\
	\IEEEauthorblockA{
		1. School of Mathematics and Statistics, Beijing Institute of Technology, Beijing 102488, P.R. China}\\
	2. Beijing Key Laboratory on MCAACI, Beijing Institute of Technology, Beijing 102488, P.R. China\\
$^{\star}$Corresponding author: li$\_$bingzhao@bit.edu.cn
}

\maketitle

\begin{abstract}
\label{abstract}
\begin{spacing}{1.25}
	We investigate the 2D quaternion windowed linear canonical transform(QWLCT) in this paper. Firstly, we propose the new definition of the QWLCT, and then several important properties of newly defined QWLCT, such as bounded, shift, modulation, orthogonality relation, are derived based on the spectral representation of the quaternionic linear canonical transform(QLCT). Finally, by the Heisenberg uncertainty
principle for the QLCT and the orthogonality relation property for the QWLCT, the Heisenberg uncertainty principle for the QWLCT is established.
\end{spacing}
\end{abstract}


\noindent  {\bf 1. Introduction.}

It is well-known that the linear canonical transform(LCT) has found broad applications in many fields of applied mathematics, signal
processing, radar system analysis, filter design, phase retrieval, and pattern recognition and optics[1,3,4,5]. All of these results show that the LCT plays an important role for chirp signal
analysis in parameter estimation, sampling  and filtering [2-12]. However, the LCT can't show the local LCT-frequency contents as a result of its global kernel. In [15,29] the LCT has been successfully applied to research the generalized the windowed function. The windowed linear canonical transform (WLCT) is a method devised to study signals whose spectral content changes with time. As shown in [14,8] some important properties of the WLCT are discussed. Those include the analogue of the Poisson summation formula, sampling formulas, Paley-Wiener theorem, and uncertainly relations. It presents the time and LCT-frequency information, and is originally a local LCT distribution.

On the other hand, some authors have generalized the linear canonical transform to quaternion-valued signals, known as the quaternionic linear canonical
transform (QLCT). The QLCT was firstly studied in [17] including prolate spheroidal wave signals and uncertainty principles [19]. Some useful properties of
the QLCT such as linearity, reconstruction formula, continuity, boundedness, positivity inversion formula and the uncertainty
principle were established in [17,13,18,20,17,21,30]. An application of the QLCT to study of generalized
swept-frequency filters was introduced in [23]. Because of the non-commutative property of multi-plication of quaternions, there are mainly three various types of 2D quaternion linear canonical transform(QLCTs): Two-sided QLCTs,
Left-sided QLCTs and Right-sided QLCTs (refer to [16]). Based on the (two-sided) QLCT,
one may extend the WLCT to quaternion algebra while possessing similar properties as in the classical case.

The WLCT has been widely used, but the quaternion windowed linear canonical transform (QWLCT) hasn't been mentioned yet.
The main goals of the present paper are to study the properties of the QLCT of 2D quaternionic signals and to derive the novel concept of the QWLCT, and
research important properties of the QWLCT such as bounded, shift, modulation, orthogonality relation. Using the Heisenberg uncertainty principle for the QLCT and the orthogonality relation property for the QWLCT,
 we establish a generalized QWLCT uncertainty principle. This also lays a good foundation for the practical application in the future.

The paper is organized as follows: Section 2 gives a brief introduction to some general
definitions and basic properties of quaternion algebra, QLCTs of 2D Quaternion-valued signals.
We give the definition and study the properties of the QWLCT in section 3.
In Section 4, provides the uncertainty principles associated with the QWLCT. We give some examples of the QWLCT in section 5.
In Section 6, some conclusions are drawn.

\noindent  {\bf 2. Preliminary.}

In this section, we mainly review some basic facts on the quaternion algebra and the QLCT, which will be needed throughout the paper.

\noindent  {\bf 2.1 Quaternion algebra.}

The quaternion algebra is an extension of the complex number to 4D algebra. It was first invented by W. R. Hamilton in 1843 and classically denoted by $\mathbb{H}$ in his honor. Every element of $\mathbb{H}$ has a Cartesian form given by
\begin{align}
\mathbb{H}=\left\lbrace q|q:=[q]_0+\mathbf{i}[q]_1+\mathbf{j}[q]_2+\mathbf{k}[q]_3, [q]_i\in\mathbb{R}, i=0,1,2,3\right\rbrace
\end{align}
where $\mathbf{i}, \mathbf{j}, \mathbf{k}$ are imaginary units obeying Hamilton's multiplication rules (see [16])
\begin{align}
\begin{split}
		&\mathbf{i}^2=\mathbf{j}^2=\mathbf{k}^2=-1,\\
&\mathbf{i}\mathbf{j}=-\mathbf{j}\mathbf{i}=\mathbf{k},\,\mathbf{j}\mathbf{k}=
-\mathbf{k}\mathbf{j}=\mathbf{i},\,\mathbf{k}\mathbf{i}
=-\mathbf{i}\mathbf{k}=\mathbf{j}.
\end{split}
\end{align}
Let $[q]_{0}$ and $q=\mathbf{i}[q]_1+\mathbf{j}[q]_2+\mathbf{k}[q]_3$ denote the real scalar part and the vector part of quaternion number $q=[q]_0+\mathbf{i}[q]_1+\mathbf{j}[q]_2+\mathbf{k}[q]_3$, respectively. Then, from [25],the
real scalar part has a cyclic multiplication symmetry
\begin{align}
\begin{split}
		[pql]_{0}=[qlp]_{0}=[lpq]_{0}, \ \  \forall q,p,l\in \mathbb{H}
\end{split}
\end{align}
the conjugate of a quaternion $q$ is defined by $\overline{q}=[q]_0-\mathbf{i}[q]_1-\mathbf{j}[q]_2-\mathbf{k}[q]_3$, and the norm of $q\in \mathbb{H}$ defined as
\begin{align}
\left| q\right|=\sqrt{q\bar{q}}=\sqrt{[q]_0^2+[q]_1^2+[q]_2^2+[q]_3^2}
\end{align}
it is easy to verify that
\begin{align}
\begin{split}
		\overline{pq}=\overline{qp}, |qp|=|q||p|, \ \ \forall q,p\in \mathbb{H}
\end{split}
\end{align}
The quaternion modules $L^2(\mathbb{R}^2,\mathbb{H})$ are defined as
\begin{align}
\begin{split}
	L^2(\mathbb{R}^2,\mathbb{H}):=\{f|f:\mathbb{R}^2\rightarrow \mathbb{H}, \|f\|_{L^2(\mathbb{R}^2,\mathbb{H})}=\int_{\mathbb{R}^2}|f(x_{1},x_{2})|^{2}{\rm d}x_1{\rm d}x_2<\infty\}	
\end{split}
\end{align}
Now we introduce an inner product of quaternion functions $f,g$ defined on $L^2(\mathbb{R}^2,\mathbb{H})$ is given by
\begin{align}
\label{inner}
	\langle f,g\rangle_{L^2(\mathbb{R}^2,\mathbb{H})}=\int_{\mathbb{R}^2}f(\mathbf{x})\overline{g(\mathbf{x})}{\rm d}\mathbf{x}, \ \ {\rm d}\mathbf{x}={\rm d}x_1{\rm d}x_2
\end{align}
with symmetric real scalar part
\begin{align}
	\langle f,g\rangle=\frac{1}{2}\left \{\langle f,g\rangle+\langle g,f\rangle\right\} =\int_{\mathbb{R}^2}[f(\mathbf{x})\overline{g(\mathbf{x})}]_{0}{\rm d}\mathbf{x}
\end{align}
The associated scalar norm of $f(\mathbf{x})\in L^2(\mathbb{R}^2,\mathbb{H})$ is defined by both (7) and (8):
\begin{align}
	\|f\|_{L^2(\mathbb{R}^2,\mathbb{H})}^2=\langle f,f\rangle_{L^2(\mathbb{R}^2,\mathbb{H})}= \int_{\mathbb{R}^2}|f(\mathbf{x})|^2 {\rm d}\mathbf{x}<\infty
\end{align}
\begin{lem}
If $f,g\in L^2(\mathbb{R}^2,\mathbb{H})$, then the Cauchy-Schwarz inequality holds[25]
\begin{align}
	\left| \langle f,g\rangle_{L^2(\mathbb{R}^2,\mathbb{H})}\right|^2\leq \|f\|_{L^2(\mathbb{R}^2,\mathbb{H})}^2\|g\|_{L^2(\mathbb{R}^2,\mathbb{H})} ^2
\end{align}
If and only if $f=-\lambda g$ for some quaternionic parameter $\lambda \in\mathbb{H}$, the equality holds. \end{lem}
\noindent  {\bf 2.2 The quaternion linear canonical transform}

The QLCT is firstly defined by Kou, et., which  is a generalization of the LCT in the frame of quaternion algebra[16]. Due to the non-commutativity of quaternion multiplication, there are three different types of the QLCT: the left-sided QLCT, the right-sided QLCT, and the two-sided QLCT. In this paper, we mainly focus on the two-sided QLCT.
\begin{defn} Let
	$A_i=\begin{bmatrix}
	a_i&b_i\\
	c_i&d_i
	\end{bmatrix}\in \mathbb{R}^{2\times 2}$ be a matrix parameter satisfying ${\rm det}(A_i)=1$, for $i=1,2$. The two-sided QLCT of signals $f\in L^2\left( \mathbb{R}^2,\mathbb{H}\right)$ is defined by
	\begin{align}
	\label{dQLCT}
	\mathcal{L}_{A_1,A_2}^{\mathbb{H}}\{f\}(\mathbf{w})=
\int_{\mathbb{R}^2}K_{A_1}^{\mathbf{i}}(x_1,\omega_1)f(\mathbf{x})K_{A_2}^{\mathbf{j}}(x_2,\omega_2)\rm{d}\mathbf{x}
	\end{align}	
	where $\mathbf{w}=(\omega_1,\omega_2)$ is regarded as the QLCT domain, and  the kernel signals $K_{A_1}^{\mathbf{i}}(x_1,\omega_1)$, $K_{A_2}^{\mathbf{j}}(x_2,\omega_2)$ are respectively given by
	\begin{align}
		\begin{split}
			K_{A_1}^{\mathbf{i}}(x_1,\omega_1):=\begin{cases}
		\frac{1}{\sqrt{2\pi b_1}}e^{\mathbf{i}\left(\frac{a_1}{2b_1}x_1^2-\frac{x_1\omega_1}{b_1}+\frac{d_1}{2b_1}\omega_1^2-\frac{\pi}{4} \right) },   &b_1\neq0  \\
		\sqrt{d_1}e^{\mathbf{i}\frac{c_1d_1}{2}\omega_1^2}\delta(x_{1}-d_{1}w_{1}),    &b_1=0
		\end{cases}
		\end{split}
	\end{align}
	and
	\begin{align}
		\begin{split}
			K_{A_2}^{\mathbf{j}}(x_2,\omega_2):=\begin{cases}
		\frac{1}{\sqrt{2\pi b_2}}e^{\mathbf{j}\left(\frac{a_2}{2b_2}x_2^2-\frac{x_2\omega_2}{b_2}+\frac{d_2}{2b_2}\omega_2^2-\frac{\pi}{4} \right) },   &b_2\neq0  \\
		\sqrt{d_2}e^{\mathbf{j}\frac{c_2d_2}{2}\omega_2^2}\delta(x_{2}-d_{2}w_{2}),    &b_2=0
		\end{cases}
		\end{split}
	\end{align}
\end{defn}
From above definition, it is noted that for $b_i=0,i=1,2$ the QLCT of a signal is essentially a chirp multiplication and it is of no particular interest for our objective in this work. Hence, without loss of generality, we set $b_i\neq0$ in the following section unless stated otherwise. Under some suitable conditions, the QLCT above is invertible and the inversion is given in the following section.
\begin{defn} Suppose $f\in L^2\left( \mathbb{R}^2,\mathbb{H}\right)$, then the inversion of the QLCT of $f$ is given by
	\begin{align}
		\begin{split}
		\label{eIQLCT}
		f(\mathbf{x})&=\mathcal{L}_{A_1,A_2}^{-1}[\mathcal{L}_{A_1,A_2}^{\mathbb{H}}\{f\}](\mathbf{x})
=\int_{\mathbb{R}^2}K_{A_1}^{-\mathbf{i}}(x_1,\omega_1)
\mathcal{L}_{A_1,A_2}^{\mathbb{H}}\left\{f\right\}(\mathbf{w})K_{A_2}^{-\mathbf{j}}(x_2,\omega_2)\rm{d}\mathbf{w}\\
		\end{split}
	\end{align}	
\end{defn}
The Parseval formula of QLCT is vary available as follows:
\begin{lem} [QLCT Parseval][17]
 Two quaternion functions $f,g\in L^2(\mathbb{R}^2,\mathbb{H})$ are related to their QLCT via
the Parseval formula, given as
\begin{align}
		\begin{split}
       \langle f,g\rangle_{L^2(\mathbb{R}^2,\mathbb{H})}=\langle \mathcal{L}_{A_1,A_2}^{\mathbb{H}}\{f\}, \mathcal{L}_{A_1,A_2}^{\mathbb{H}}\{g\}\rangle_{L^2(\mathbb{R}^2,\mathbb{H})}
        \end{split}
	\end{align}
For $f=g$, one have
\begin{align}
		\begin{split}
\| f\|^{2}_{L^2(\mathbb{R}^2,\mathbb{H})}=\|\mathcal{L}_{A_1,A_2}^{\mathbb{H}}\{f\}\|^{2}_{L^2(\mathbb{R}^2,\mathbb{H})}
        \end{split}
	\end{align}
\end{lem}

\noindent{\bf 3. Quaternionic Windowed Linear Canonical Transform (QWLCT)}

In this section, the generalization of window function associated with the QLCT will be discussed, which is denoted by QWLCT. Moreover, several basic
properties of them are investigated.

\medskip

\noindent{\bf 3.1. The definition of the 2D QWLCT}

This section leads to the 2D quaternion window function associated with the QLCT. Due to the noncommutative property of
multiplication of quaternions, there are three different types of
QWLCTs: Two-sided QWLCT, Left-sided QWLCT and Right-sided QWLCT. Alternatively, we use the (Two-sided)
QWLCT to define the QWLCT.

\noindent  {\bf Definition 3.1 (QWLCT)}
	Let
	$A_i=\begin{bmatrix}
	a_i&b_i\\
	c_i&d_i
	\end{bmatrix}\in \mathbb{R}^{2\times 2}$ be a matrix parameter satisfying ${\rm det}(A_i)=1$, for $i=1,2$. $\phi\in L^2(\mathbb{R}^2,\mathbb{H})\backslash \{0\}$ be a quaternion window function. The two-sided QWLCT of a signal $f\in L^2\left( \mathbb{R}^2,\mathbb{H}\right)$ with respect to $\phi$ is defined by
	\begin{align}
	\label{dQLCT}
	G^{A_{1},A_{2}}_{\phi}f(\mathbf{w,u})&=\int_{\mathbb{R}^2}K_{A_1}^{\mathbf{i}}(x_1,\omega_1)f(\mathbf{x}) \overline{\phi(\mathbf{x-u})} K_{A_2}^{\mathbf{j}}(x_2,\omega_2)\rm{d}\mathbf{x}
	\end{align}	
where $\mathbf{u}=(u_1,u_2)\in\mathbb{R}^2$, $K_{A_1}^{\mathbf{i}}(x_1,\omega_1)$ and $K_{A_2}^{\mathbf{j}}(x_2,\omega_2)$ are given by (12) and (13), respectively.

For a fixed $\mathbf{u}$, we have
\begin{align}
		\begin{split}
		G^{A_{1},A_{2}}_{\phi}f(\mathbf{w,u})=\mathcal{L}_{A_1,A_2}^{\mathbb{H}}
{\{f(\mathbf{x})\overline{\phi(\mathbf{x-u})}\}(\mathbf{w})} \\
		\end{split}
	\end{align}
Applying the inverse QLCT to (18), we have
\begin{align}
		\begin{split}
		f(\mathbf{x})\overline{\phi(\mathbf{x-u})}&=\mathcal{L}_{A_1,A_2}^{-1}\{G^{A_1,A_2}_{\phi}f(\mathbf{w,u})\} =\int_{\mathbb{R}^2}\overline{K_{A_1}^{\mathbf{i}}(x_1,\omega_1)}
G^{A_1,A_2}_{\phi}\{f\}(\mathbf{w,u})\overline{K_{A_2}^{\mathbf{j}}(x_2,\omega_2)} \rm{d}\mathbf{w}
		\end{split}
	\end{align}
\medskip

\noindent{\bf 3.2. Some properties of QWLCT}
\bigskip

In this subsection, we discuss several basic properties of the QWLCT. These properties play important roles in signal representation.
\begin{thm}[Boundedness] Let $\phi\in L^2(\mathbb{R}^2,\mathbb{H})\backslash \{0\}$ be a window function and $f\in L^2\left( \mathbb{R}^2,\mathbb{H}\right)$, then
\begin{align}
		\begin{split}
		|G^{A_1,A_2}_{\phi}f(\mathbf{w,u})|\leq \frac{1}{2\pi\sqrt{|b_{1}b_{2}|}} \|f\|_{L^{2}(\mathbb{R}^2)} \|\phi\|_{L^{2}(\mathbb{R}^2)}\\
		\end{split}
	\end{align} \end{thm}
\begin{proof} By using the quaternion Cauchy-Schwarz inequality (10)
\begin{align}
		\begin{split}
		|G^{A_1,A_2}_{\phi}f(\mathbf{w,u})|^{2}&= \left|\int_{\mathbb{R}^2}K_{A_1}^{\mathbf{i}}(x_1,\omega_1)f(\mathbf{x}) \overline{\phi(\mathbf{x-u})}K_{A_2}^{\mathbf{j}}(x_2,\omega_2)\rm{d}\mathbf{x}\right| ^2\\
&\leq \left( \int_{\mathbb{R}^2} \left| K_{A_1}^{\mathbf{i}}(x_1,\omega_1)f(\mathbf{x}) \overline{\phi(\mathbf{x-u})}K_{A_2}^{\mathbf{j}}(x_2,\omega_2)\right| \rm{d}\mathbf{x}\right) ^2 \\
&=\left( \frac{1}{\sqrt{4\pi^{2}|b_{1}b_{2}|}}\int_{\mathbb{R}^2} \left| f(\mathbf{x}) \overline{\phi(\mathbf{x-u})}\right| \rm{d}\mathbf{x}\right) ^2 \\
&\leq \frac{1}{4\pi^{2}|b_{1}b_{2}|}\left(\int_{\mathbb{R}^2} \left| f(\mathbf{x})\right|^{2}\rm{d}\mathbf{x}\right)\left(\int_{\mathbb{R}^2} \left| \overline{\phi(\mathbf{x-u})}\right|^{2} \rm{d}\mathbf{x}\right)\\
&=\frac{1}{4\pi^{2}|b_{1}b_{2}|} \|f\|_{L^{2}(\mathbb{R}^2)}^{2} \|\phi\|_{L^{2}(\mathbb{R}^2)}^{2}\\
		\end{split}
	\end{align}
where applying the change of variables $\mathbf{x-u}=\mathbf{t}$ in the last step. Then we have
\begin{align}
		\begin{split}
		|G^{A_1,A_2}_{\phi}f(\mathbf{w,u})|\leq \frac{1}{2\pi\sqrt{|b_{1}b_{2}|}} \|f\|_{L^{2}(\mathbb{R}^2)} \|\phi\|_{L^{2}(\mathbb{R}^2)}\\
		\end{split}
	\end{align}
which completes the proof.\end{proof}
\begin{thm}[Linearity] Let $\phi\in L^2(\mathbb{R}^2,\mathbb{H})\backslash \{0\}$ be a window function and $f,g \in L^2\left( \mathbb{R}^2,\mathbb{H}\right)$, the QWLCT is a linear operator, namely,
  \begin{align}
        \begin{split}
		[G^{A_1,A_2}_{\phi}(\lambda f+\mu g)](\mathbf{w,u})= \lambda G^{A_1,A_2}_{\phi}f(\mathbf{w,u})+\mu G^{A_1,A_2}_{\phi}g(\mathbf{w,u})\\
		\end{split}
	\end{align}
for arbitrary constants $\lambda$ and $\mu$. \end{thm}
\begin{proof} This follows directly from the linearity of the product and the integration
involved in Definition 3.1.\end{proof}
\begin{thm}[Parity]
Let $\phi\in L^2(\mathbb{R}^2,\mathbb{H})\backslash \{0\}$ be a window function and $f \in L^2\left( \mathbb{R}^2,\mathbb{H}\right)$. Then we have
\begin{align}
        \begin{split}
		G^{A_1,A_2}_{P\phi}\{P f\}(\mathbf{w,u})= G^{A_1,A_2}_{\phi}f(\mathbf{-w,-u})\\
		\end{split}
	\end{align}
where $P\phi(\mathbf{x})=\phi(-\mathbf{x})$ for every $\phi\in L^2(\mathbb{R}^2)$. \end{thm}
\begin{proof} A direct calculation gives, for every $f \in L^2\left( \mathbb{R}^2,\mathbb{H}\right)$,
\begin{align}
        \begin{split}
		G^{A_1,A_2}_{P\phi}\{P f\}(\mathbf{w,u})&= \int_{\mathbb{R}^2}K_{A_1}^{\mathbf{i}}(x_1,\omega_1)f(\mathbf{-x}) \overline{\phi(\mathbf{-(x-u)})} K_{A_2}^{\mathbf{j}}(x_2,\omega_2)\rm{d}\mathbf{x}\\
&=\int_{\mathbb{R}^2}\frac{1}{\sqrt{2\pi b_1}}e^{\mathbf{i}\left(\frac{a_1}{2b_1}(-x_1)^2-\frac{(-x_1)(-\omega_1)}{b_1}+\frac{d_1}{2b_1}(-\omega_1)^2-\frac{\pi}{4} \right) }\\
&\times f(\mathbf{-x}) \overline{\phi(\mathbf{-(x-u)})}\frac{1}{\sqrt{2\pi b_2}}e^{\mathbf{j}\left(\frac{a_2}{2b_2}(-x_2)^2-\frac{(-x_2)(-\omega_2)}{b_2}+\frac{d_2}{2b_2}(-\omega_2)^2-\frac{\pi}{4} \right) }\rm{d}\mathbf{x}\\
		\end{split}
	\end{align}
which proves the theorem according to Definition 3.1.\end{proof}
\begin{thm}[Shift]
 Let $\phi\in L^2(\mathbb{R}^2,\mathbb{H})\backslash \{0\}$ be a window function and $f \in L^2\left( \mathbb{R}^2,\mathbb{H}\right)$. Then we have
\begin{align}
        \begin{split}
		G^{A_1,A_2}_{\phi}\{T_{\mathbf{r}}f\}(\mathbf{w,u})= e^{\mathbf{i}r_{1}\omega_{1}c_{1}}e^{-\mathbf{i}\frac{a_{1}r_{1}^{2}}{2}c_{1}}
G^{A_1,A_2}_{\phi}\{f\}(\mathbf{m,n}) e^{\mathbf{j}r_{2}\omega_{2}c_{2}}e^{-\mathbf{j}\frac{a_{2}r_{2}^{2}}{2}c_{2}}\\
		\end{split}
	\end{align}
where $T_{\mathbf{r}}f(\mathbf{x})=f(\mathbf{x-r}), \mathbf{r}=(r_{1},r_{2})$,
 $\mathbf{m}=(m_{1},m_{2})$, $ \mathbf{n}=(n_{1},n_{2})\in \mathbb{R}^2$, $ m_{i}=w_{i}-a_{i}r_{i} $, $ n_{i}=u_{i}-r_{i} $, $i=1,2$.\end{thm}
\begin{proof} By Definition 3.1, we have
\begin{align}
        \begin{split}
		G^{A_1,A_2}_{\phi}\{T_{\mathbf{r}}f\}(\mathbf{w,u})= \int_{\mathbb{R}^2}K_{A_1}^{\mathbf{i}}(x_1,\omega_1)f(\mathbf{x-r}) \overline{\phi(\mathbf{x-u})} K_{A_2}^{\mathbf{j}}(x_2,\omega_2)\rm{d}\mathbf{x}\\
		\end{split}
	\end{align}
By making the change of variable $\mathbf{t}=\mathbf{x-r}$ in the above expression, we obtain
\begin{align}
        \begin{split}
		G^{A_1,A_2}_{\phi}\{T_{\mathbf{r}}f\}(\mathbf{w,u})&= \int_{\mathbb{R}^2}\frac{1}{\sqrt{2\pi b_1}}e^{\mathbf{i}\left(\frac{a_1}{2b_1}(r_1+t_{1})^2-\frac{(r_1+t_{1})
\omega_1}{b_1}+\frac{d_1}{2b_1}\omega_1^2-\frac{\pi}{4} \right) }
 f(\mathbf{t}) \overline{\phi(\mathbf{t-(u-r)})} \\
&\times\frac{1}{\sqrt{2\pi b_2}}e^{\mathbf{j}\left(\frac{a_2}{2b_2}(r_2+t_{2})^2-\frac{(r_2+t_{2})\omega_2}{b_2}+\frac{d_2}{2b_2}
\omega_2^2-\frac{\pi}{4} \right) }\rm{d}\mathbf{t}\\
&=\frac{1}{\sqrt{2\pi b_1}}e^{\mathbf{i}\frac{a_{1}}{2b_{1}}r_{1}^{2}-\mathbf{i}\frac{r_{1}\omega_{1}}{b_{1}}} \int_{\mathbb{R}^2}e^{\mathbf{i}\left(\frac{a_1}{2b_1}t_{1}^2-\frac{(\omega_{1}-r_{1}a_{1})t_1}{b_1}+\frac{d_1}{2b_1}
(\omega_1-r_{1}a_{1}+r_{1}a_{1})^2-\frac{\pi}{4}\right) }\\
&\times f(\mathbf{t}) \overline{\phi(\mathbf{t-(u-r)})}\frac{1}{\sqrt{2\pi b_2}} e^{\mathbf{j}\frac{a_{2}}{2b_{2}}r_{2}^{2}-\mathbf{j}\frac{r_{2}\omega_{2}}{b_{2}}}\rm{d}\mathbf{t}
e^{\mathbf{j}\left(\frac{a_2}{2b_2}t_{2}^2-\frac{(\omega_{2}-r_{2}a_{2})t_2}{b_2}+\frac{d_2}{2b_2}
(\omega_2-r_{2}a_{2}+r_{2}a_{2})^2-\frac{\pi}{4}\right) }\\
&=e^{\mathbf{i}\frac{d_{1}}{2b_{1}}(2r_{1}a_{1}(\omega_{1}-r_{1}a_{1})+(r_{1}a_{1})^{2})}
e^{\mathbf{i}\frac{a_{1}}{2b_{1}}r_{1}^{2}}e^{-\mathbf{i}\frac{r_{1}\omega_{1}}{b_{1}}}
\int_{\mathbb{R}^2}\frac{1}{\sqrt{2\pi b_1}}e^{\mathbf{i}\left(\frac{a_1}{2b_1}t_{1}^2-\frac{t_{1}(\omega_1-r_{1}a_{1})}{b_1}+\frac{d_1}{2b_1}
(\omega_1-r_{1}a_{1})^2-\frac{\pi}{4} \right) }\\
&\times f(\mathbf{t}) \overline{\phi(\mathbf{t-(u-r)})}\frac{1}{\sqrt{2\pi b_2}}e^{\mathbf{j}\left(\frac{a_2}{2b_2}t_{2}^2-\frac{t_{2}(\omega_2-r_{2}a_{2})}{b_2}+\frac{d_2}{2b_2}
(\omega_2-r_{2}a_{2})^2-\frac{\pi}{4} \right) }\rm{d}\mathbf{t}\\
&\times e^{\mathbf{j}\frac{d_{2}}{2b_{2}}(2r_{2}a_{2}(\omega_{2}-r_{2}a_{2})+(r_{2}a_{2})^{2})}
e^{\mathbf{j}\frac{a_{2}}{2b_{2}}r_{2}^{2}}e^{-\mathbf{j}\frac{r_{2}\omega_{2}}{b_{2}}}\\
		\end{split}
	\end{align}
Applying the definition of the QWLCT, the above expression can be rewritten in the form
\begin{align}
        \begin{split}
		G^{A_1,A_2}_{\phi}\{T_{\mathbf{r}}f\}(\mathbf{w,u})&=e^{\mathbf{i}\frac{d_{1}}{2b_{1}}(2r_{1}a_{1}
(\omega_{1}-r_{1}a_{1})+(r_{1}a_{1})^{2})}
e^{\mathbf{i}\frac{a_{1}}{2b_{1}}r_{1}^{2}}e^{-\mathbf{i}\frac{r_{1}\omega_{1}}{b_{1}}} G^{A_1,A_2}_{\phi}\{f\}(\mathbf{m,n})\\
&\times e^{\mathbf{j}\frac{d_{2}}{2b_{2}}(2r_{2}a_{2}(\omega_{2}-r_{2}a_{2})+(r_{2}a_{2})^{2})}
e^{\mathbf{j}\frac{a_{2}}{2b_{2}}r_{2}^{2}}e^{-\mathbf{j}\frac{r_{2}\omega_{2}}{b_{2}}}\\
\end{split}
	\end{align}
Because $a_{i}d_{i}-b_{i}c_{i}=1$, for $i=1,2$. We get
\begin{align}
        \begin{split}
		e^{\mathbf{i}\frac{d_{1}}{2b_{1}}(2r_{1}a_{1}(\omega_{1}-r_{1}a_{1})+(r_{1}a_{1})^{2})}
e^{\mathbf{i}\frac{a_{1}}{2b_{1}}r_{1}^{2}}e^{-\mathbf{i}\frac{r_{1}\omega_{1}}{b_{1}}}=
e^{\mathbf{i}r_{1}\omega_{1}c_{1}}e^{-\mathbf{i}\frac{a_{1}r_{1}^{2}}{2}c_{1}}\\
\end{split}
	\end{align}
\begin{align}
        \begin{split}
		e^{\mathbf{j}\frac{d_{2}}{2b_{2}}(2r_{2}a_{2}(\omega_{2}-r_{2}a_{2})+(r_{2}a_{2})^{2})}
e^{\mathbf{j}\frac{a_{2}}{2b_{2}}r_{2}^{2}}e^{-\mathbf{j}\frac{r_{2}\omega_{2}}{b_{2}}}=
e^{\mathbf{j}r_{2}\omega_{2}c_{2}}e^{-\mathbf{j}\frac{a_{2}r_{2}^{2}}{2}c_{2}}\\
\end{split}
	\end{align}
We finally arrive at
\begin{align}
        \begin{split}
		G^{A_1,A_2}_{\phi}\{T_{\mathbf{r}}f\}(\mathbf{w,u})&=e^{\mathbf{i}r_{1}\omega_{1}c_{1}}
e^{-\mathbf{i}\frac{a_{1}r_{1}^{2}}{2}c_{1}}
G^{A_1,A_2}_{\phi}\{f\}(\mathbf{m,n})e^{\mathbf{j}r_{2}\omega_{2}c_{2}}e^{-\mathbf{j}\frac{a_{2}r_{2}^{2}}{2}c_{2}}\\
\end{split}
	\end{align}
which completes the proof.\end{proof}
\begin{thm}[Modulation]
Let $\phi\in L^2(\mathbb{R}^2,\mathbb{H})\backslash \{0\}$ be a window function and $f \in L^2\left( \mathbb{R}^2,\mathbb{H}\right)$.
$\mathbb{M}_{s}f$ be modulation operator defined by $\mathbb{M}_{s}f(\mathbf{x})=e^{\mathbf{i}x_{1}s_{1}}f(\mathbf{x})e^{\mathbf{j}x_{2}s_{2}}$ with $\mathbf{s}=(s_{1},s_{2})\in \mathbb{R}^2$.
Then we have
\begin{align}
        \begin{split}
		G^{A_1,A_2}_{\phi}\{\mathbb{M}_{s}f\}(\mathbf{w,u})&=e^{\mathbf{i}\omega_{1}s_{1}d_{1}}e^{-\mathbf{i}\frac{b_{1}d_{1}s_{1}^{2}}{2}}
G^{A_1,A_2}_{\phi}\{f\}(\mathbf{v,u})e^{\mathbf{j}\omega_{2}s_{2}d_{2}}e^{-\mathbf{j}\frac{b_{2}d_{2}s_{2}^{2}}{2}}
\end{split}
	\end{align}
\end{thm}
where $\mathbf{v}=(v_{1},v_{2})\in \mathbb{R}^2$,
$ v_{i}=w_{i}-s_{i}b_{i} $, $i=1,2$.
\begin{proof} From Definition 3.1, it follows that
\begin{align}
        \begin{split}
		G^{A_1,A_2}_{\phi}\{\mathbb{M}_{s}f\}(\mathbf{w,u})&=\int_{\mathbb{R}^2}\frac{1}{\sqrt{2\pi b_1}}e^{\mathbf{i}\left(\frac{a_1}{2b_1}x_1^2-\frac{x_1\omega_1}{b_1}+\frac{d_1}{2b_1}\omega_1^2-\frac{\pi}{4} \right) }e^{\mathbf{i}x_{1}s_{1}} f(\mathbf{x}) \overline{\phi(\mathbf{x-u})}
 e^{\mathbf{j}x_{2}s_{2}} \frac{1}{\sqrt{2\pi b_2}}e^{\mathbf{j}\left(\frac{a_2}{2b_2}x_2^2-\frac{x_2\omega_2}{b_2}+\frac{d_2}{2b_2}\omega_2^2-\frac{\pi}{4} \right) }\rm{d}\mathbf{x}\\
&=\frac{1}{\sqrt{2\pi b_1}}\int_{\mathbb{R}^2}e^{\mathbf{i}
\left(\frac{a_1}{2b_1}x_1^2-\frac{x_1(\omega_1-s_{1}b_{1})}{b_1}+\frac{d_1}{2b_1}
((\omega_1-s_{1}b_{1})+s_{1}b_{1})^2-\frac{\pi}{4} \right)} f(\mathbf{x}) \overline{\phi(\mathbf{x-u})} \\
&\times \frac{1}{\sqrt{2\pi b_2}}e^{\mathbf{j}
\left(\frac{a_2}{2b_2}x_2^2-\frac{x_2(\omega_2-s_{2}b_{2})}{b_2}+\frac{d_2}{2b_2}
((\omega_2-s_{2}b_{2})+s_{2}b_{2})^2-\frac{\pi}{4} \right) }\rm{d}\mathbf{x}\\
&=e^{\mathbf{i}\omega_{1}s_{1}d_{1}}e^{-\mathbf{i}\frac{b_{1}d_{1}s_{1}^{2}}{2}}\int_{\mathbb{R}^2}\frac{1}{\sqrt{2\pi b_1}}e^{\mathbf{i}
\left(\frac{a_1}{2b_1}x_1^2-\frac{x_1(\omega_1-s_{1}b_{1})}{b_1}+\frac{d_1}{2b_1}(\omega_1-s_{1}b_{1})^2-\frac{\pi}{4} \right)}\\
&\times f(\mathbf{x}) \overline{\phi(\mathbf{x-u})} e^{\mathbf{j}\omega_{2}s_{2}d_{2}}e^{-\mathbf{j}\frac{b_{2}d_{2}s_{2}^{2}}{2}} \frac{1}{\sqrt{2\pi b_2}}e^{\mathbf{j}
\left(\frac{a_2}{2b_2}x_2^2-\frac{x_2(\omega_2-s_{2}b_{2})}{b_2}+\frac{d_2}{2b_2}(\omega_2-s_{2}b_{2})^2-\frac{\pi}{4} \right) }\rm{d}\mathbf{x}\\
\end{split}
	\end{align}
Hence,
\begin{align}
        \begin{split}
		G^{A_1,A_2}_{\phi}\{\mathbb{M}_{s}f\}(\mathbf{w,u})&=e^{\mathbf{i}\omega_{1}s_{1}d_{1}}e^{-\mathbf{i}\frac{b_{1}d_{1}s_{1}^{2}}{2}}
G^{A_1,A_2}_{\phi}\{f\}(\mathbf{v,u})e^{\mathbf{j}\omega_{2}s_{2}d_{2}}e^{-\mathbf{j}\frac{b_{2}d_{2}s_{2}^{2}}{2}}
\end{split}
	\end{align}
which completes the proof.\end{proof}
\begin{thm}[Inversion formula]
 Let $\phi\in L^2(\mathbb{R}^2,\mathbb{H})\backslash \{0\}$ be a window function, $ 0<\|\phi\|^{2}<\infty $ and $f \in L^2\left( \mathbb{R}^2,\mathbb{H}\right)$.
Then we have the inversion formula of the QWLCT,
\begin{align}
        \begin{split}
		f(\mathbf{x})=\frac{1}{\|\phi\|^{2}}\int_{\mathbb{R}^2}\int_{\mathbb{R}^2}\overline{K_{A_1}^{\mathbf{i}}(x_1,\omega_1)}
G^{A_1,A_2}_{\phi}\{f\}(\mathbf{w,u})\overline{K_{A_2}^{\mathbf{j}}(x_2,\omega_2)} \phi\mathbf{(x-u)}\rm{d}\mathbf{w}\rm{d}\mathbf{u}\\
\end{split}
	\end{align} \end{thm}
\begin{proof}  Multiplying both sides of (19) from the right by $ \phi(\mathbf{x-u})$ and integrating with respect to $ \rm{d}\mathbf{u}$ we get
\begin{align}
		\begin{split}
		\int_{\mathbb{R}^2}f(\mathbf{x})\overline{\phi(\mathbf{x-u})}\phi(\mathbf{x-u})\rm{d}\mathbf{u}
&=\int_{\mathbb{R}^2}\int_{\mathbb{R}^2}\overline{K_{A_1}^{\mathbf{i}}(x_1,\omega_1)}
G^{A_1,A_2}_{\phi}\{f\}(\mathbf{w,u})\overline{K_{A_2}^{\mathbf{j}}(x_2,\omega_2)} \phi\mathbf{(x-u)}\rm{d}\mathbf{w}\rm{d}\mathbf{u}\\
		\end{split}
	\end{align}
Using (9), we have
\begin{align}
        \begin{split}
		f(\mathbf{x})=\frac{1}{\|\phi\|^{2}}\int_{\mathbb{R}^2}\int_{\mathbb{R}^2}\overline{K_{A_1}^{\mathbf{i}}(x_1,\omega_1)}
G^{A_1,A_2}_{\phi}\{f\}(\mathbf{w,u})\overline{K_{A_2}^{\mathbf{j}}(x_2,\omega_2)} \phi\mathbf{(x-u)}\rm{d}\mathbf{w}\rm{d}\mathbf{u}\\
\end{split}
	\end{align}
which completes the proof.\end{proof}
\begin{thm}[Orthogonality relation]
 Let $\phi,\psi \in L^2(\mathbb{R}^2,\mathbb{H})\backslash \{0\}$ be a window function and $f,g \in L^2\left( \mathbb{R}^2,\mathbb{H}\right)$.
Then
\begin{align}
		\begin{split}
		\langle G^{A_1,A_2}_{\phi}\{f\}(\mathbf{w,u}), {G^{A_1,A_2}_{\psi}\{g\}(\mathbf{w,u})} \rangle
=[\langle f,g\rangle\langle\phi,\psi\rangle]_{0}\\
		\end{split}
	\end{align} \end{thm}
\begin{proof} Using (3),(8), we get the output as follows:

\begin{align}
		\begin{split}
		\langle G^{A_1,A_2}_{\phi}\{f\}(\mathbf{w,u}), {G^{A_1,A_2}_{\psi}\{g\}(\mathbf{w,u})} \rangle&=\int_{\mathbb{R}^4}\left[ G^{A_1,A_2}_{\phi}\{f\}(\mathbf{w,u}) \overline{{G^{A_1,A_2}_{\psi}\{g\}(\mathbf{w,u})}}\right]_{0}\rm{d}\mathbf{w}\rm{d}\mathbf{u}\\
&=\int_{\mathbb{R}^4}\left[ G^{A_1,A_2}_{\phi}\{f\}(\mathbf{w,u})
 \overline{\int_{\mathbb{R}^2}K_{A_1}^{\mathbf{i}}(x_1,\omega_1)g(\mathbf{x}) \overline{\psi(\mathbf{x-u})} K_{A_2}^{\mathbf{j}}(x_2,\omega_2)\rm{d}\mathbf{x}}\right]_{0}\rm{d}\mathbf{w}\rm{d}\mathbf{u}\\
&=\int_{\mathbb{R}^6}\left[ G^{A_1,A_2}_{\phi}\{f\}(\mathbf{w,u}) K_{A_2}^{\mathbf{-j}}(x_2,\omega_2)\psi(\mathbf{x-u})\overline{g(\mathbf{x})}K_{A_1}^{\mathbf{-i}}(x_1,\omega_1) \right]_{0}\rm{d}\mathbf{w}\rm{d}\mathbf{u}\rm{d}\mathbf{x}\\
&=\int_{\mathbb{R}^6}\left[ K_{A_1}^{\mathbf{-i}}(x_1,\omega_1) G^{A_1,A_2}_{\phi}\{f\}(\mathbf{w,u}) K_{A_2}^{\mathbf{-j}}(x_2,\omega_2)\psi(\mathbf{x-u})\overline{g(\mathbf{x})} \right]_{0}\rm{d}\mathbf{w}\rm{d}\mathbf{u}\rm{d}\mathbf{x}\\
&=\int_{\mathbb{R}^4}\left[ \int_{\mathbb{R}^2}K_{A_1}^{\mathbf{-i}}(x_1,\omega_1) G^{A_1,A_2}_{\phi}\{f\}(\mathbf{w,u}) K_{A_2}^{\mathbf{-j}}(x_2,\omega_2)\rm{d}\mathbf{w}\psi(\mathbf{x-u})\overline{g(\mathbf{x})} \right]_{0}\rm{d}\mathbf{u}\rm{d}\mathbf{x}\\
		\end{split}
	\end{align}
Because
\begin{align}
\overline{K_{A_1}^{\mathbf{i}}(x_1,\omega_1)}=K_{A_{1}}^{\mathbf{i}^{-1}}(\omega_1,x_1)=K_{A_{1}^{-1}}^{\mathbf{i}}(\omega_1,x_1)\\
\overline{K_{A_2}^{\mathbf{j}}(x_2,\omega_2)}=K_{A_{2}}^{\mathbf{j}^{-1}}(\omega_2,x_2)=K_{A_{2}^{-1}}^{\mathbf{j}}(\omega_2,x_2)
	\end{align}	
Using (19), we have
\begin{align}
		\begin{split}
\langle G^{A_1,A_2}_{\phi}\{f\}(\mathbf{w,u}), {G^{A_1,A_2}_{\psi}\{g\}(\mathbf{w,u})} \rangle&=\int_{\mathbb{R}^4}\left[\int_{\mathbb{R}^2}K_{A_1^{-1}}^{\mathbf{i}}(\omega_1,x_1) G^{A_1,A_2}_{\phi}\{f\}(\mathbf{w,u}) K_{A_2^{-1}}^{\mathbf{j}}(\omega_2,x_2)\rm{d}\mathbf{w} \psi(\mathbf{x-u})\overline{g(\mathbf{x})}\right]_{0}\rm{d}\mathbf{u}\rm{d}\mathbf{x}\\
&=\int_{\mathbb{R}^4}\left[ f(\mathbf{x})\overline{\phi(\mathbf{x-u})} \psi(\mathbf{x-u})\overline{g(\mathbf{x})}  \right]_{0}\rm{d}\mathbf{u}\rm{d}\mathbf{x}\\
		\end{split}
	\end{align}
Using the change of variables $ \mathbf{x-u}=\mathbf{y}$, the equation becomes
\begin{align}
		\begin{split}
		\langle G^{A_1,A_2}_{\phi}\{f\}(\mathbf{w,u}), {G^{A_1,A_2}_{\psi}\{g\}(\mathbf{w,u})} \rangle
&=\int_{\mathbb{R}^4}\left[ f(\mathbf{x})\overline{\phi(\mathbf{y})} \psi(\mathbf{y})\overline{g(\mathbf{x})}  \right]_{0}\rm{d}\mathbf{y}\rm{d}\mathbf{x}\\
&=\left[\int_{\mathbb{R}^2} f(\mathbf{x})\overline{g(\mathbf{x})}\rm{d}\mathbf{x}  \int_{\mathbb{R}^2}\psi(\mathbf{y}) \overline{\phi(\mathbf{y})} \rm{d}\mathbf{y} \right]_{0}\\
&=[\langle f,g\rangle\langle\phi,\psi\rangle]_{0}
		\end{split}
	\end{align}
which completes the proof.\end{proof}
Based on the above theorem, we may conclude the following important consequences.

(i) If $\phi=\psi$, then
\begin{align}
		\begin{split}
		\langle G^{A_1,A_2}_{\phi}\{f\}(\mathbf{w,u}), {G^{A_1,A_2}_{\phi}\{g\}(\mathbf{w,u})} \rangle
=\|\phi\|^{2}_{L^2(\mathbb{R}^2)}\langle f,g\rangle
		\end{split}
	\end{align}

(ii) If $f=g$, then
\begin{align}
		\begin{split}
		\langle G^{A_1,A_2}_{\phi}\{f\}(\mathbf{w,u}), {G^{A_1,A_2}_{\psi}\{f\}(\mathbf{w,u})} \rangle
=\|f\|^{2}_{L^2(\mathbb{R}^2)}\langle \phi,\psi\rangle
		\end{split}
	\end{align}

(iii) If  $f=g$ and $\phi=\psi$, then
\begin{align}
		\begin{split}
		\langle G^{A_1,A_2}_{\phi}\{f\}(\mathbf{w,u}), {G^{A_1,A_2}_{\phi}\{f\}(\mathbf{w,u})} \rangle&=
\int_{\mathbb{R}^2}\int_{\mathbb{R}^2}| G^{A_1,A_2}_{\phi}\{f\}(\mathbf{w,u})|^{2}\rm{d}\mathbf{w}\rm{d}\mathbf{u}
=\|\textit{f}\|^{2}_{L^2(\mathbb{R}^2)} \|\phi\|^{2}_{L^2(\mathbb{R}^2)}
		\end{split}
	\end{align}
Table 1:
Properties of the QWLCT of $\phi\in L^2(\mathbb{R}^2,\mathbb{H})\backslash \{0\}$  and $f, g\in L^2\left( \mathbb{R}^2,\mathbb{H}\right)$,
where $\lambda, \mu \in\mathbb{H}$ are arbitrary constants.

\begin{tabular}{l*{6}{c}r}
\hline
Property    & Function & QWLCT  \\
\hline
Boundedness   & $ |G^{A_1,A_2}_{\phi}f(\mathbf{w,u})|\leq$ &$\frac{1}{2\pi\sqrt{|b_{1}b_{2}|}} \|f\|_{L^{2}(\mathbb{R}^2)} \|\phi\|_{L^{2}(\mathbb{R}^2)} $\\
Linearity     & $\lambda f+\mu g$ & $\lambda G^{A_1,A_2}_{\phi}f(\mathbf{w,u})+\mu G^{A_1,A_2}_{\phi}g(\mathbf{w,u})$\\
Parity        & $ G^{A_1,A_2}_{P\phi}\{P f\}(\mathbf{w,u}) $ & $ G^{A_1,A_2}_{\phi}f(\mathbf{-w,-u}) $  \\
Shift         & $ f(\mathbf{x-r}) $ & $ e^{\mathbf{i}r_{1}\omega_{1}c_{1}}e^{-\mathbf{i}\frac{a_{1}r_{1}^{2}}{2}c_{1}}
G^{A_1,A_2}_{\phi}\{f\}(\mathbf{m,n})$\\
 &&$\times e^{\mathbf{j}r_{2}\omega_{2}c_{2}}e^{-\mathbf{j}\frac{a_{2}r_{2}^{2}}{2}c_{2}} $  \\
Modulation    & $ e^{\mathbf{i}x_{1}s_{1}}f(\mathbf{x})e^{\mathbf{j}x_{2}s_{2}} $  & $ e^{\mathbf{i}\omega_{1}s_{1}d_{1}}e^{-\mathbf{i}\frac{b_{1}d_{1}s_{1}^{2}}{2}}
G^{A_1,A_2}_{\phi}\{f\}(\mathbf{v,u})$\\
 &&$\times e^{\mathbf{j}\omega_{2}s_{2}d_{2}}e^{-\mathbf{j}\frac{b_{2}d_{2}s_{2}^{2}}{2}}  $ \\
\hline
\end{tabular}

\begin{tabular}{l*{6}{c}r}
\hline
Formula      \\
\hline
Inversion         & $ f(\mathbf{x})=\frac{1}{\|\phi\|^{2}}\int_{\mathbb{R}^2}\int_{\mathbb{R}^2}\overline{K_{A_1}^{\mathbf{i}}(x_1,\omega_1)}
G^{A_1,A_2}_{\phi}\{f\}(\mathbf{w,u})$\\
 &$\times\overline{K_{A_2}^{\mathbf{j}}(x_2,\omega_2)} \phi\mathbf{(x-u)}\rm{d}\mathbf{w}\rm{d}\mathbf{u} $ \\
Orthogonality     & $ \langle G^{A_1,A_2}_{\phi}\{f\}(\mathbf{w,u}), {G^{A_1,A_2}_{\psi}\{g\}(\mathbf{w,u})} \rangle
=[\langle f,g\rangle\langle\phi,\psi\rangle]_{0} $   \\
& $\langle G^{A_1,A_2}_{\phi}\{f\}(\mathbf{w,u}), {G^{A_1,A_2}_{\phi}\{g\}(\mathbf{w,u})} \rangle
=\|\phi\|^{2}_{L^2(\mathbb{R}^2)}\langle f,g\rangle$  \\
& $\langle G^{A_1,A_2}_{\phi}\{f\}(\mathbf{w,u}), {G^{A_1,A_2}_{\psi}\{f\}(\mathbf{w,u})} \rangle
=\|f\|^{2}_{L^2(\mathbb{R}^2)}\langle \phi,\psi\rangle$ \\
& $\langle G^{A_1,A_2}_{\phi}\{f\}(\mathbf{w,u}), {G^{A_1,A_2}_{\phi}\{f\}(\mathbf{w,u})} \rangle$\\
&$=\int_{\mathbb{R}^2}\int_{\mathbb{R}^2}| G^{A_1,A_2}_{\phi}\{f\}(\mathbf{w,u})|^{2}\rm{d}\mathbf{w}\rm{d}\mathbf{u}$$=\|f\|^{2}_{L^2(\mathbb{R}^2)} \|\phi\|^{2}_{L^2(\mathbb{R}^2)}$ \\
\hline
\end{tabular}
\medskip

\noindent{\bf 4. Heisenberg's Uncertainty Principle for the QWLCT}

The Heisenberg uncertainty principle and quaternions are both basic for
quantum mechanics. In quantum mechanics an uncertainty principle asserts
that one cannot make certain of the position and velocity of
an electron (or any particle) at the same time. That is, increasing
the knowledge of the position decreases the knowledge
of the velocity or momentum of an electron.
In signal processing an uncertainty principle states that the product of the variances of the signal in the
time and frequency domains has a lower bound. The uncertainty principles of the Fourier
transform and the quaternion Fourier
transform had been studied in [24,22]. Recently, the authors established the uncertainty principles associated with the LCT in [31,25,32,26,33,34]. The uncertainty principles for the WLCT had been discussed in [27]. Theirs uncertainties
were generalizations of Lieb's uncertainty principles in the WLCT domains. Recently, Heisenberg's uncertainty relations were extended to the quaternion linear canonical transform [17,16]. This uncertainty principle prescribes a lower bound on the product
of the effective widths of quaternion-valued signals in the spatial and frequency domains. Let us give a short and simple proof the Heisenberg's Uncertainty Principle of the QWLCT.
\begin{lem}[QLCT uncertainty principle] [16]
 Let $f \in L^2\left( \mathbb{R}^2,\mathbb{H}\right)$ and \\
 $ \mathcal{L}_{A_1,A_2}^{\mathbb{H}}\{f\}(\mathbf{w})\in L^2\left( \mathbb{R}^2,\mathbb{H}\right)$
then we have
\begin{align}
		\begin{split}
\int_{\mathbb{R}^2}x^{2}_{k}|f(\mathbf{x})|^{2}\rm{d}\mathbf{x}\int_{\mathbb{R}^2}\omega^{2}_{k}
|\mathcal{L}_{A_1,A_2}^{\mathbb{H}}\{\textit{f}\}(\mathbf{w})|^{2}\rm{d}\mathbf{w}\geq \frac{b_{k}^{2}}{4}\left(\int_{\mathbb{R}^2}|\textit{f}(\mathbf{x})|^{2}\rm{d}\mathbf{x}\right)^{2},
		\end{split}
	\end{align}
where $k=1,2.$
\end{lem}
Substituting the inverse transform for the QLCT (14) into the left-hand side of
(48), we obtain
\begin{align}
		\begin{split}
\int_{\mathbb{R}^2}x^{2}_{k}|\mathcal{L}_{A_1,A_2}^{-1}[\mathcal{L}_{A_1,A_2}^{\mathbb{H}}\{f\}]|^{2}
\rm{d}\mathbf{x}\int_{\mathbb{R}^2}\omega^{2}_{k}
|\mathcal{L}_{A_1,A_2}^{\mathbb{H}}\{\textit{f}\}(\mathbf{w})|^{2}\rm{d}\mathbf{w}
\geq \frac{b_{k}^{2}}{4}\left(\int_{\mathbb{R}^2}|f(\mathbf{x})|^{2}\rm{d}\mathbf{x}\right)^{2},
		\end{split}
	\end{align}
where $k=1,2.$ Further, applying Plancherel's theorem for the QLCT (16) to the right-hand side of
(49), we have
\begin{align}
		\begin{split}
\int_{\mathbb{R}^2}x^{2}_{k}|\mathcal{L}_{A_1,A_2}^{-1}[\mathcal{L}_{A_1,A_2}^{\mathbb{H}}\{f\}]|^{2}
\rm{d}\mathbf{x}\int_{\mathbb{R}^2}\omega^{2}_{k}
|\mathcal{L}_{A_1,A_2}^{\mathbb{H}}\{\textit{f}\}(\mathbf{w})|^{2}\rm{d}\mathbf{w}
\geq \frac{b_{k}^{2}}{4}\left(\int_{\mathbb{R}^2}|\mathcal{L}_{A_1,A_2}^{\mathbb{H}}\{f\}(\mathbf{w})|^{2}
\rm{d}\mathbf{w}\right)^{2}, k=1,2
		\end{split}
	\end{align}
\begin{lem}
 Let $\phi \in L^2(\mathbb{R}^2,\mathbb{H})\backslash \{0\}$ be a window function and $f \in L^2\left( \mathbb{R}^2,\mathbb{H}\right)$.
Then
\begin{align}
		\begin{split}
\|\phi\|^{2}_{L^2(\mathbb{R}^2)} \int_{\mathbb{R}^2}x^{2}_{k}|f(\mathbf{x})|^{2}\rm{d}\mathbf{x}=
\int_{\mathbb{R}^2} \int_{\mathbb{R}^2}x^{2}_{k}|\mathcal{L}_{A_1,A_2}^{-1}\{G_{\phi}^{A_1,A_2}
\textit{f}(\mathbf{w,u})\}(\mathbf{x})|^{2}\rm{d}\mathbf{x}\rm{d}\mathbf{u}
		\end{split}
	\end{align}\end{lem}
\begin{proof} Using $Fubini^{,}s$  theorem and (19),we get
\begin{align}
		\begin{split}
\|\phi\|^{2}_{L^2(\mathbb{R}^2)} \int_{\mathbb{R}^2}x^{2}_{k}|f(\mathbf{x})|^{2}\rm{d}\mathbf{x}&=
\int_{\mathbb{R}^2} x^{2}_{k}|f(\mathbf{x})|^{2}\rm{d}\mathbf{x} \int_{\mathbb{R}^2}|\phi(\mathbf{x-u})|^{2}\rm{d}\mathbf{u}\\
&=\int_{\mathbb{R}^2}\int_{\mathbb{R}^2}x^{2}_{k}|f(\mathbf{x})|^{2}|\phi(\mathbf{x-u})|^{2}\rm{d}\mathbf{x}\rm{d}\mathbf{u}\\
&=\int_{\mathbb{R}^2}\int_{\mathbb{R}^2}x^{2}_{k}|f(\mathbf{x})\overline{\phi(\mathbf{x-u})}|^{2}\rm{d}\mathbf{x}\rm{d}\mathbf{u}\\
&=\int_{\mathbb{R}^2} \int_{\mathbb{R}^2}x^{2}_{k}|\mathcal{L}_{A_1,A_2}^{-1}\{G_{\phi}^{A_1,A_2}f(\mathbf{w,u})\}(\mathbf{x})|^{2}\rm{d}\mathbf{x}\rm{d}\mathbf{u}
		\end{split}
	\end{align} \end{proof}
Now we arrive at the following important result.
\begin{thm}[QWLCT uncertainty principle]
 Let $\phi \in L^2(\mathbb{R}^2,\mathbb{H})\backslash \{0\}$ be a window function and $G^{A_1,A_2}_{\phi}\{f\} \in L^2\left( \mathbb{R}^2,\mathbb{H}\right)$ be the QWLCT of $f$. Then for every $f \in L^2\left( \mathbb{R}^2,\mathbb{H}\right)$ we have the following inequality
\begin{align}
		\begin{split}
\left(\int_{\mathbb{R}^2} \int_{\mathbb{R}^2}\omega^{2}_{k}
|G_{\phi}^{A_1,A_2}f(\mathbf{w,u})|^{2}\rm{d}\mathbf{w}\rm{d}\mathbf{u}\right )^{\frac{1}{2}} \left(\int_{\mathbb{R}^2}x^{2}_{k}|f(\mathbf{x})|^{2}\rm{d}\mathbf{x}\right )^{\frac{1}{2}}
\geq\frac{b_{k}}{2}\|f\|_{L^2(\mathbb{R}^2)}^{2}\|\phi\|_{L^2(\mathbb{R}^2)}
		\end{split}
	\end{align}\end{thm}
\begin{proof}
Assume that $ \mathcal{L}_{A_1,A_2}^{\mathbb{H}}\{f\}\in L^2\left( \mathbb{R}^2,\mathbb{H}\right)$.
Since $G^{A_1,A_2}_{\phi}f \in L^2\left( \mathbb{R}^2,\mathbb{H}\right)$
we can replace the QLCT of $f$ by the QWLCT of $f$ on the both sides of (50). We have
\begin{align}
		\begin{split}
\int_{\mathbb{R}^2}x^{2}_{k}|\mathcal{L}_{A_1,A_2}^{-1}[G_{\phi}^{A_1,A_2}
f(\mathbf{w,u})]|^{2}\rm{d}\mathbf{x}\int_{\mathbb{R}^2}\omega^{2}_{k}
|G_{\phi}^{A_1,A_2}\textit{f}(\mathbf{w,u})|^{2}\rm{d}\mathbf{w}
\geq \frac{b_{k}^{2}}{4}\left(\int_{\mathbb{R}^2}|G_{\phi}^{A_1,A_2}f(\mathbf{w,u})|^{2}\rm{d}\mathbf{w}\right)^{2}
		\end{split}
	\end{align}
Taking the square root on both sides of (54) and integrating both sides with respect
to $ \rm{d}\mathbf{u}$ , we have
\begin{align}
		\begin{split}
\int_{\mathbb{R}^2}\left\{\left(\int_{\mathbb{R}^2}x^{2}_{k}|\mathcal{L}_{A_1,A_2}^{-1}[G_{\phi}^{A_1,A_2}
f(\mathbf{w,u})]|^{2}\rm{d}\mathbf{x}\right)^{\frac{1}{2}}
\left(\int_{\mathbb{R}^2}\omega^{2}_{k}
|G_{\phi}^{A_1,A_2}f(\mathbf{w,u})|^{2}\rm{d}\mathbf{w}\right)^{\frac{1}{2}}\right\}\rm{d}\mathbf{u}\\
\geq \frac{b_{k}}{2}\int_{\mathbb{R}^2}\int_{\mathbb{R}^2}|G_{\phi}^{A_1,A_2}f(\mathbf{w,u})|^{2}
\rm{d}\mathbf{w}\rm{d}\mathbf{u}\\
		\end{split}
	\end{align}
Furthermore, applying the Cauchy-Schwarz inequality to the left-hand side
of (55), we have
\begin{align}
		\begin{split}
\left(\int_{\mathbb{R}^2}\int_{\mathbb{R}^2}x^{2}_{k}|\mathcal{L}_{A_1,A_2}^{-1}[G_{\phi}^{A_1,A_2}
f(\mathbf{w,u})]|^{2}\rm{d}\mathbf{x}\rm{d}\mathbf{u}\right)^{\frac{1}{2}}
\left(\int_{\mathbb{R}^2}\int_{\mathbb{R}^2}\omega^{2}_{k}
|G_{\phi}^{A_1,A_2}f(\mathbf{w,u})|^{2}\rm{d}\mathbf{w}\rm{d}\mathbf{u}\right)^{\frac{1}{2}}\\
 \geq\frac{b_{k}}{2}\int_{\mathbb{R}^2}\int_{\mathbb{R}^2}|G_{\phi}^{A_1,A_2}
 f(\mathbf{w,u})|^{2}\rm{d}\mathbf{w}\rm{d}\mathbf{u}
		\end{split}
	\end{align}
Inserting (51) into the second term on the left-hand side of (56) and substituting
(47) into the right-hand side of this inequality. Then, we have
\begin{align}
		\begin{split}
\left(\|\phi\|^{2}_{L^2(\mathbb{R}^2)} \int_{\mathbb{R}^2}x^{2}_{k}|f(\mathbf{x})|^{2}\rm{d}\mathbf{x}\right)^{\frac{1}{2}} \left(\int_{\mathbb{R}^2}\int_{\mathbb{R}^2}\omega^{2}_{k}
|G_{\phi}^{A_1,A_2}f(\mathbf{w,u})|^{2}\rm{d}\mathbf{w}\right)^{\frac{1}{2}}
 \geq \frac{b_{k}}{2}\|f\|^{2}_{L^2(\mathbb{R}^2)} \|\phi\|^{2}_{L^2(\mathbb{R}^2)}
		\end{split}
	\end{align}
Dividing both sides of (57) by $\|\phi\|_{L^2(\mathbb{R}^2)}$, we obtain the desired result.\end{proof}

\medskip

\noindent{\bf 5. Examples of the QWLCT}

Given the window function of the two-dimensional Haar function defined by
\begin{align}
		\begin{split}
			\phi(\mathbf{x})=\begin{cases}
		1,   &0\leq x_{1}<\frac{1}{2}, 0\leq x_{2}<\frac{1}{2} \\
		-1,    &\frac{1}{2}\leq x_{1}<1 , \frac{1}{2}\leq x_{2}<1  \\
         0,   &otherwise,
		\end{cases}
		\end{split}
	\end{align}
Consider a 2D Gaussian quaternionic
function of the form $ f(x_{1},x_{2})=e^{-(\alpha_{1}x_{1}^{2}+\alpha_{2}x_{2}^{2})}$, for
$ \alpha_{1}, \alpha_{2}\in \mathbb{R}$  are positive real constants.

Then the QWLCT of $ f $ is given by
\begin{align}
		\begin{split}
G^{A_{1},A_{2}}_{\phi}f(\mathbf{w,u})&=\int_{\mathbb{R}^2}\frac{1}{\sqrt{2\pi b_1}}e^{\mathbf{i}\left(\frac{a_1}{2b_1}x_1^2-\frac{x_1u_1}{b_1}+\frac{d_1}{2b_1}u_1^2-\frac{\pi}{4} \right) }f(\mathbf{x}) \overline{\phi(\mathbf{x-u})}\times\frac{1}{\sqrt{2\pi b_2}}e^{\mathbf{j}\left(\frac{a_2}{2b_2}x_2^2-\frac{x_2u_2}{b_2}+\frac{d_2}{2b_2}u_2^2-\frac{\pi}{4} \right) }\rm{d}\mathbf{x}\\
&=\int^{\frac{1}{2}+u_{1}}_{u_{1}}\frac{1}{\sqrt{2\pi b_1}}e^{\mathbf{i}\left(\frac{a_1}{2b_1}x_1^2-\frac{x_1u_1}{b_1}+\frac{d_1}{2b_1}u_1^2-\frac{\pi}{4} \right) }e^{-\alpha_{1}x_{1}^{2}}\rm{d}\textit{x}_{1}\int^{\frac{1}{2}+u_{2}}_{u_{2}}\frac{1}{\sqrt{2\pi b_2}}e^{\mathbf{j}\left(\frac{a_2}{2b_2}x_2^2-\frac{x_2u_2}{b_2}+\frac{d_2}{2b_2}u_2^2-\frac{\pi}{4} \right) }e^{-\alpha_{2}x_{2}^{2}}\rm{d}\textit{x}_{2}\\
&-\int^{1+u_{1}}_{\frac{1}{2}+u_{1}}\frac{1}{\sqrt{2\pi b_1}}e^{\mathbf{i}\left(\frac{a_1}{2b_1}x_1^2-\frac{x_1u_1}{b_1}+\frac{d_1}{2b_1}u_1^2-\frac{\pi}{4} \right) }e^{-\alpha_{1}x_{1}^{2}}\rm{d}\textit{x}_{1}\int^{1+u_{2}}_{\frac{1}{2}+u_{2}}\frac{1}{\sqrt{2\pi b_2}}e^{\mathbf{j}\left(\frac{a_2}{2b_2}x_2^2-\frac{x_2u_2}{b_2}+\frac{d_2}{2b_2}u_2^2-\frac{\pi}{4} \right) }e^{-\alpha_{2}x_{2}^{2}}\rm{d}\textit{x}_{2}\\
		\end{split}
	\end{align}
By completing squares, we have
\begin{align}
		\begin{split}
G^{A_{1},A_{2}}_{\phi}f(\mathbf{w,u})&=\int^{\frac{1}{2}+u_{1}}_{u_{1}}\frac{1}{\sqrt{2\pi b_1}}e^{-\left(\sqrt{\frac{2b_{1}\alpha_{1}-\mathbf{i}a_{1}}{2b_{1}}}x_{1}+\frac{\mathbf{i}
\omega_{1}}{\sqrt{2b_{1}(2b_{1}\alpha_{1}-\mathbf{i}a_{1})}} \right)^{2}} e^{-\frac{\omega_{1}^{2}}{2b_{1}\alpha_{1}-\mathbf{i}a_{1}}+\frac{\mathbf{i}d_{1}}{2b_{1}}
\omega_{1}^{2}-\frac{\pi}{4}}\rm{d}\textit{x}_{1}\\
&\times \int^{\frac{1}{2}+u_{2}}_{u_{2}}\frac{1}{\sqrt{2\pi b_2}}e^{-\left(\sqrt{\frac{2b_{2}\alpha_{2}-\mathbf{j}a_{2}}{2b_{2}}}x_{2}+\frac{\mathbf{j}
\omega_{2}}{\sqrt{2b_{2}(2b_{2}\alpha_{2}-\mathbf{j}a_{2})}} \right)^{2}} e^{-\frac{\omega_{2}^{2}}{2b_{2}\alpha_{2}-\mathbf{j}a_{2}}+\frac{\mathbf{j}d_{2}}{2b_{2}}
\omega_{2}^{2}-\frac{\pi}{4}}\rm{d}\textit{x}_{2}\\
&-\int^{1+u_{1}}_{\frac{1}{2}+u_{1}}\frac{1}{\sqrt{2\pi b_1}}e^{-\left(\sqrt{\frac{2b_{1}\alpha_{1}-\mathbf{i}a_{1}}{2b_{1}}}x_{1}+\frac{\mathbf{i}
\omega_{1}}{\sqrt{2b_{1}(2b_{1}\alpha_{1}-\mathbf{i}a_{1})}} \right)^{2}} e^{-\frac{\omega_{1}^{2}}{2b_{1}\alpha_{1}-\mathbf{i}a_{1}}+\frac{\mathbf{i}d_{1}}{2b_{1}}
\omega_{1}^{2}-\frac{\pi}{4}}\rm{d}\textit{x}_{1}\\
&\times \int^{1+u_{2}}_{\frac{1}{2}+u_{2}}\frac{1}{\sqrt{2\pi b_2}}e^{-\left(\sqrt{\frac{2b_{2}\alpha_{2}-\mathbf{j}a_{2}}{2b_{2}}}x_{2}+\frac{\mathbf{j}
\omega_{2}}{\sqrt{2b_{2}(2b_{2}\alpha_{2}-\mathbf{j}a_{2})}} \right)^{2}} e^{-\frac{\omega_{2}^{2}}{2b_{2}\alpha_{2}-\mathbf{j}a_{2}}+\frac{\mathbf{j}d_{2}}{2b_{2}}
\omega_{2}^{2}-\frac{\pi}{4}}\rm{d}\textit{x}_{2}\\
		\end{split}
	\end{align}
Making the substitutions $ A_{1}=\sqrt{\frac{2b_{1}\alpha_{1}-\mathbf{i}a_{1}}{2b_{1}}},  A_{2}=\sqrt{\frac{2b_{2}\alpha_{2}-\mathbf{j}a_{2}}{2b_{2}}},
B_{1}=\frac{\mathbf{i}\omega_{1}}{\sqrt{2b_{1}(2b_{1}\alpha_{1}-\mathbf{i}a_{1})}}, B_{2}=\frac{\mathbf{j}\omega_{2}}{\sqrt{2b_{2}(2b_{2}\alpha_{2}-\mathbf{j}a_{2})}},\\
J_{1}= e^{-\frac{\omega_{1}^{2}}{2b_{1}\alpha_{1}-\mathbf{i}a_{1}}+\frac{\mathbf{i}d_{1}}{2b_{1}}\omega_{1}^{2}-\frac{\pi}{4}}$ and
$J_{2}= e^{-\frac{\omega_{2}^{2}}{2b_{2}\alpha_{2}-\mathbf{j}a_{2}}+\frac{\mathbf{j}d_{2}}{2b_{2}}\omega_{2}^{2}-\frac{\pi}{4}}$
in the above expression we immediately obtain
\begin{align}
		\begin{split}
G^{A_{1},A_{2}}_{\phi}f(\mathbf{w,u})&=\int^{\frac{1}{2}+u_{1}}_{u_{1}}\frac{1}{\sqrt{2\pi b_1}}e^{-\left(A_{1}x_{1}+B_{1} \right)^{2}}J_{1}\rm{d}\textit{x}_{1}
\int^{\frac{1}{2}+u_{2}}_{u_{2}}\frac{1}{\sqrt{2\pi b_2}}e^{-\left(A_{2}x_{2}+B_{2} \right)^{2}}J_{2}\rm{d}\textit{x}_{2}\\
&-\int^{1+u_{1}}_{\frac{1}{2}+u_{1}}\frac{1}{\sqrt{2\pi b_1}}e^{-\left(A_{1}x_{1}+B_{1} \right)^{2}}J_{1}\rm{d}\textit{x}_{1}\int^{1+u_{2}}_{\frac{1}{2}+u_{2}}\frac{1}{\sqrt{2\pi b_2}}e^{-\left(A_{2}x_{2}+B_{2} \right)^{2}}J_{2}\rm{d}\textit{x}_{2}\\
		\end{split}
	\end{align}
Substituting $ y_{1}=A_{1}x_{1}+B_{1} $ and $ y_{2}=A_{2}x_{2}+B_{2}$ in the above equation, we have
\begin{align}
		\begin{split}
G^{A_{1},A_{2}}_{\phi}f(\mathbf{w,u})
&=\frac{J_{1}}{A_{1}\sqrt{2\pi b_{1}}} \int^{A_{1}(\frac{1}{2}+u_{1})+B_{1}}_{A_{1}u_{1}+B_{1}}e^{-y_{1}^{2}}\rm{d}\textit{y}_{1}\times
\frac{J_{2}}{A_{2}\sqrt{2\pi b_{2}}} \int^{A_{2}(\frac{1}{2}+u_{2})+B_{2}}_{A_{2}u_{2}+B_{2}}e^{-y_{2}^{2}}\rm{d}\textit{y}_{2}\\
&-\frac{J_{1}}{A_{1}\sqrt{2\pi b_{1}}} \int^{A_{1}(1+u_{1})+B_{1}}_{A_{1}(\frac{1}{2}+u_{1})+B_{1}}e^{-y_{1}^{2}}\rm{d}\textit{x}_{1}\times
\frac{J_{2}}{A_{2}\sqrt{2\pi b_{2}}}  \int^{A_{2}(1+u_{2})+B_{2}}_{A_{2}(\frac{1}{2}+u_{2})+B_{2}}e^{-y_{2}^{2}}\rm{d}\textit{x}_{2}\\
&=\frac{J_{1}}{A_{1}\sqrt{2\pi b_{1}}}
\times \left(\int^{A_{1}u_{1}+B_{1}}_{0}(-e^{-y_{1}^{2}})\rm{d}\textit{y}_{1}+
\int^{A_{1}(\frac{1}{2}+u_{1})+B_{1}}_{0}e^{-y_{1}^{2}}\rm{d}\textit{y}_{1}\right)\\
&\times \frac{J_{2}}{A_{2}\sqrt{2\pi b_{2}}}
\times \left(\int^{A_{2}u_{2}+B_{2}}_{0}(-e^{-y_{2}^{2}})\rm{d}\textit{y}_{2}+
\int^{A_{2}(\frac{1}{2}+u_{2})+B_{2}}_{0}e^{-y_{2}^{2}}\rm{d}\textit{y}_{2}\right)\\
&-\frac{J_{1}}{A_{1}\sqrt{2\pi b_{1}}}
\times \left(\int^{A_{1}(\frac{1}{2}+u_{1})+B_{1}}_{0}(-e^{-y_{1}^{2}})\rm{d}\textit{y}_{1}+
\int^{A_{1}(1+u_{1})+B_{1}}_{0}e^{-y_{1}^{2}}\rm{d}\textit{y}_{1}\right)\\
&\times \frac{J_{2}}{A_{2}\sqrt{2\pi b_{2}}}\times\left(\int^{A_{2}(\frac{1}{2}+u_{2})+B_{2}}_{0}(-e^{-y_{2}^{2}})\rm{d}\textit{y}_{2}+
\int^{A_{2}(1+u_{2})+B_{2}}_{0}e^{-y_{2}^{2}}\rm{d}\textit{y}_{2}\right)
		\end{split}
	\end{align}

Equation (62) can be written in the form
\begin{align}
		\begin{split}
G^{A_{1},A_{2}}_{\phi}f(\mathbf{w,u})&=
\frac{J_{1}}{2\sqrt{2b_{1}\alpha_{1}-\mathbf{i}\alpha_{1}}}\times \left(-erf(A_{1}u_{1}+B_{1})+erf(A_{1}(\frac{1}{2}+u_{1})+B_{1})\right)\\
&\times \frac{J_{2}}{2\sqrt{2b_{2}\alpha_{2}-\mathbf{j}\alpha_{2}}}\times \left(-erf(A_{2}u_{2}+B_{2})+erf(A_{2}(\frac{1}{2}+u_{2})+B_{2})\right)\\
&-\frac{J_{1}}{2\sqrt{2b_{1}\alpha_{1}-\mathbf{i}\alpha_{1}}} \times\left(-erf(A_{1}(\frac{1}{2}+u_{1})+B_{1})+erf(A_{1}(1+u_{1})+B_{1})\right)\\
&\times \frac{J_{2}}{2\sqrt{2b_{2}\alpha_{2}-\mathbf{j}\alpha_{2}}}\times \left(-erf(A_{2}(\frac{1}{2}+u_{2})+B_{2})+erf(A_{2}(1+u_{2})+B_{2})\right)\\
		\end{split}
	\end{align}

where $erf(x)=\frac{2}{\sqrt{\pi}}\int^{x}_{0}e^{-t^{2}}\rm{d}\textit{t}$.

The WLCT plays a vital role in  the analysis of signals [8,35]. The windowed linear canonical kernel provides a unique representation for
the signals so that the WLCT is highly effective. It is natural to think whether the QWLCT can also be applied to such problems. For illustrative purposes, we shall give the example above.

\medskip

\noindent  {\bf 6. Conclusions}

Due to the non-commutativity of quaternion multiplication, there are three different types of the QLCT: the left-sided QLCT, the
right-sided QLCT, and the two-sided QLCT.
Using the basic concepts of quaternion algebra and the two-sided QLCT we introduced the the two-sided QWLCT.
Important properties of the QWLCT such as boundedness, linearity, parity, shift, modulation, inversion formula and orthogonality relation were derived.
Based on the QWLCT properties and the uncertainty principle for the QLCT, a new uncertainty principle for the QWLCT have been established.
The results in this paper are new in the literature. Further
investigations on this topic are now under investigation such as the left-sided QWLCT, the
right-sided QWLCT. They will be reported in a forthcoming paper.
\medskip

\noindent  {\bf Acknowledge}

This work is supported by the National Natural Science Foundation of China (No. 61671063).


\end{document}